\documentclass{amsart}
\usepackage{graphicx,enumerate,textcomp}
\usepackage{amsmath,amssymb,amsfonts,dsfont,amsthm}
\usepackage{theoremref}
\usepackage{psfrag}
\usepackage{euscript}
\usepackage{mathrsfs}
\usepackage{hyperref}
\usepackage[foot]{amsaddr}
\usepackage{setspace}
\usepackage{biblatex}
\usepackage[utf8]{inputenc}
\numberwithin{equation}{section}

\begin{document}
\title[Dispersive Decay for Kawahara]{Dispersive decay bound of small data solutions to Kawahara equation in a finite time scale}
\author[Jongwon Lee]{Jongwon Lee}
\address{Department of Mathematical Sciences, Korea Advaced Institute of Science and Technology, 291 Daehak-ro Yuseong-gu, Daejeon 34141, South Korea}
\email{jwlee18math@kaist.ac.kr}

\maketitle
\newtheorem{thm}{Theorem}[section]

\newtheorem{cor}[thm]{Corollary}

\newtheorem{lem}[thm]{Lemma}

\newtheorem{prop}[thm]{Proposition}

\theoremstyle{plain}

\theoremstyle{definition}
\newtheorem{rmk}[thm]{Remark}
\newtheorem{defn}[thm]{Definition}
\newtheorem{example}[thm]{Example}
\begin{abstract}
	In this article, we prove that small localized data yield solutions to Kawahara type equations which have linear dispersive decay on a finite time, depending on the size of the initial data. We use the similar method used to derive the dispersive decay bound of the solutions to the KdV equation, with some steps being simpler. This result is expected to be the first result of the small data global bounds of the fifth-order dispersive equations with quadratic nonlinearity.
\end{abstract}
\section{Introduction}
A global behavior of a small data solution to a nonlinear dispersive equation has drawn attentions of many authors. It might result in the finite or infinite time linear decay bound, or the linear or modified scattering. The goal of this paper is to find a dispersive decay bound of a small data solution to the Kawahara equation:\begin{equation}
    u_t-u_{5x}=uu_x.\tag{KW}\label{eq:KW}\end{equation}
We also consider the case of the modified Kawahara equation \begin{equation*}
u_t-u_{5x}=cu^2u_x,
\tag{mKW}\label{eq:mKW}  
\end{equation*} with $c=\pm 1,$ which is the special case of the fifth-order KdV equation \begin{equation}\tag{5KdV}\label{eq:5KdV}
    u_t-u_{5x}=c_1uu_{3x}+c_2u_xu_{2x}+c_3u^2u_x.
\end{equation}
The equation \eqref{eq:KW} arises from the study of capillary waves on a shallow layer and magneto-sound propagation in
plasmas; see \cite{kawahara1972oscillatory}.

The local and global Cauchy problems of the fifth-order nonlinear dispersive equations such as \eqref{eq:KW}, \eqref{eq:mKW}, \eqref{eq:5KdV} have been extensively studied. Kenig, Ponce, and Vega \cite{kenig1994higher, kenig1994hierarchy} proved the local well-posedness (LWP) of the general class of dispersive equations \[\partial_tu+\partial_x^{2j+1}u+P(u,\partial_xu,\cdots,\partial_x^{2j}u)=0\] where $P$ is a polynomial without constant or linear terms, on the space $H^s(\mathbb{R})\cap L^2(|x|^mdx)$ with $s,m$ being a sufficiently large positive integer. Kwon \cite{kwon2008fifth} proved LWP of \eqref{eq:5KdV} in $H^s$ with $s>5/2$ without the weighted $L^2$ structure. Kenig and Pilod \cite{kenig2012well} proved the global well-posedness (GWP) of the equations \eqref{eq:5KdV} with $c_2=2c_1$ and \eqref{eq:KW} in $H^s$ with $s\geq 2.$ Very recently, Bringmann et al. \cite{bringmann2019global} proved GWP in $H^{-1}(\mathbb{R})$ of \eqref{eq:5KdV} with $c_1=10,\;c_2=20,\;c_3=30$.
\\
When it comes to \eqref{eq:KW}, Cui, Deng, and Tao \cite{cui2006global} proved LWP in $H^s$ with $s>-1$ and GWP in $L^2$. This result is later refined by Wang, Cui, and Deng \cite{wang2007global} down to the regularity $s\geq -\frac{7}{5}$ for LWP and $s>-\frac{1}{2}$ for GWP. Chen et al. \cite{chen2009low} proved LWP of both \eqref{eq:KW} and \eqref{eq:mKW}, with the regularity being $s>-\frac{7}{4}$ for \eqref{eq:KW} and $s\geq -\frac{1}{4}$ for \eqref{eq:mKW}. Chen and Guo \cite{chen2011globalwell} extended the LWP result to $s\geq -\frac{7}{4}$ and also proved GWP for the same regularity. Kato \cite{kato2013global} proved LWP for $s\geq -2$ and GWP for $s\geq -\frac{38}{21}.$ He also proved that the LWP threshold is optimal by proving ill-posedness for $s<-2.$

Also, the small data global analysis of nonlinear dispersive equations has been spotlighted for a long time. For instance, Deift and Zhou \cite{deift1993steepest} proved the small data asymptotics of mKdV first, taking advantage of the complete integrability of the equation. Hayashi and Naumkin \cite{hayashi2001modified, hayashi2006final, hayashi1999large} proved the similar result without using the complete integrability, which helped them extend the result for nonlinearities with time-dependent coefficients. However, their results required an initial data to have zero mean. \cite{cazenave2003semilinear} introduces some classical scattering results for defocusing NLS, together with asymptotic completeness. It also mentions that the one-dimensional cubic NLS \[iu_t+u_{xx}=\lambda|u|^2u,\quad \lambda=\pm1\tag{3NLS}\label{eq:3NLS}\] does not show a linear scattering. Rather, the modified scattering results for \eqref{eq:3NLS} and similar types of equations are proved by many authors with different settings.  Ifrim and Tataru \cite{ifrim2015global} proved the global dynamics of \eqref{eq:3NLS} by testing the wave packets. Adopting this idea of Ifrim and Tataru, Harrop-Griffiths \cite{harrop2016long} proved the modified scattering and asymptotic completeness results for the mKdV equation. The work of Harrop-Griffiths is meaningful in that it requires neither the integrability of the equation, nor the zero mean of the initial datum. Okamoto \cite{okamoto2018long} later proved the global behavior of solutions to the fifth-order mKdV type equation, without proving asymptotic completeness, using the same idea as above. Okamoto later proved \cite{okamoto2019asymptotic} the asymptotic behavior of solutions to the higher order KdV-type equation with critical nonliearity. Here the nonlinearity is said to be critical in the sense that the small data solution shows the modified scattering behavior at that order of nonlinearity.

Unlike the mKdV-type equations, the global dynamics of KdV-type equations have not been studied so much. The main difficulty of the global analysis of KdV-type equations is that the nonlinearities are less perturbative than those of the equations on the mKdV hierarchy; the nonlinearities of mKdV-type equations are cubic, whereas those of KdV-type equations are quadratic. The first result regarding the nonlinear dispersive equation with a quadratic nonlinearity is due to \cite{ifrim2017well}, which discovered that the solution to the Benjamin-Ono equation \[u_t+Hu_{xx}=uu_x\] has linear dispersive decay bound in an almost global time scale $|t|\sim e^{\frac{c}{\epsilon}}$ where $\epsilon$ is the size of an initial datum. A global result for KdV equation is discovered fairly recently, by Ifrim et al., \cite{ifrim2019dispersive} which stated that the linear dispersive decay bound may break down beyond the quartic time scale $|t|\sim \epsilon^{-3}$ where $\epsilon$ is the size of an initial datum. They proved that the result is optimal, i.e. the time scale cannot be extended further, using the inverse scattering theory. No other meaningful results are known about the global dynamics of KdV-type equations. The main goal of this paper is to show the first result regarding the global behavior of the fifth-order dispersive equations with quadratic nonlinearities, whose byproduct is the small data global behavior of \eqref{eq:mKW}.
\subsection{Main result}
First we state the main result of this paper:
\begin{thm}
Consider the Cauchy problem \begin{equation*}
\begin{cases}u_t-u_{5x}=\pm u^mu_x,\\u(0,x)=u_0(x),\end{cases}\quad m=1,2.
\end{equation*} Suppose that the initial datum $u_0$ satisfies \begin{equation*}
    \Vert u_0\Vert_{\dot{B}^{-\frac{1}{2}}_{2,\infty}}+\Vert xu_0\Vert_{\dot{H}^{\frac{1}{2}}}\leq\epsilon\ll 1, 
\end{equation*} then we have the following pointwise bound \[|\partial_x^k u(t,x)|\lesssim\epsilon t^{-\frac{1}{8}-\frac{k}{4}}\langle x\rangle^{-\frac{3}{8}+\frac{k}{4}},\quad k=0,1,2,3\] provided $|t|\ll \epsilon^{-\frac{5}{5-2m}}.$\\ Moreover, in the elliptic region $\{x\gg t^{\frac{1}{5}}\},$ we have a better bound \[|\partial_x^k u(t,x)|\lesssim\epsilon t^{-\frac{k}{4}}\langle x\rangle^{-1+\frac{k}{4}}\log(t^{-\frac{1}{5}}\langle x\rangle),\quad k=0,1,2,3\] when $|t|\ll \epsilon^{-\frac{5}{5-2m}}.$
\end{thm}
\begin{rmk}Note that the case $m=1$ corresponds to \eqref{eq:KW}, and $m=2$ to \eqref{eq:mKW}.
\end{rmk}
\begin{rmk}
 The time scale in \eqref{eq:KW} is worse than that in the KdV equation, which is even worse than that in the Benjamin-Ono equation. This can be deduced heuristically as follows: The group velocity of the fifth-order linear equation is given as $\omega(\xi)=-5\xi^4$, so the high frequency part of the solution travels much faster towards to the negative $x$-axis than in the second or third-order equations. However, the possible emergence of nonlinear ``bulk", such as solitons and dispersive shocks, inturrupts such linear dispersive decay. The faster the linear part travels, the earlier the nonlinear bulk emerges.
\\The time scale in \eqref{eq:mKW} is better than that of the KdV and \eqref{eq:KW}. This is because the nonlinearity becomes cubic, making it more perturbative. However, such time scale is worse than that of the mKdV equation, as the mKdV equation has modified scattering and global dispersive decay bound; see \cite{harrop2016long}. Such discrepancy arises from the same heuristics used in the comparision between \eqref{eq:KW}, KdV, and Benjamin-Ono equations. One may also compare this result with \cite{okamoto2018long} or \cite{okamoto2019asymptotic}, in which case the nonlinearities become more perturbative than \eqref{eq:mKW}.
\end{rmk}
\begin{rmk}
The two norms used in the main result are inspired from \cite{ifrim2019dispersive}. As seen in the later analysis, the $\dot{B}^{-\frac{1}{2}}_{2,\infty}$ norm helps properly control the low frequency, and the growth of the high frequency norms is slow enough that it does not disrupt our analysis. The $\dot{H}^{\frac{1}{2}}$ norm exactly does the opposite. It does the control of the high frequency, and guarantees the mild growth of the low frequency part.
\end{rmk}

\section{Preliminaries}
\subsection{Definitions and Notations}
Let $1\leq p<\infty.$ $L^p(X)$ be the set of measurable function on $X$ with the norm $\Vert f\Vert_{L^p(X)}=\left(\int_X|f|^pd\mu\right)^{1/p}$ on the measure space $(X,\mu).$

If $p=\infty,$ then $\Vert f\Vert_{L^\infty(X)}=\inf\{M>0:|f(x)|\leq M \;\;\text{a.e.}\}$. If $X=\mathbb{R}$ and $\mu$ is the Lebesgue measure, then we write $L^p(X)=L^p$ for simplicity, in either case $p<\infty$ or $p=\infty.$

We use the time-dependent Japanese bracket $\langle x\rangle=(x^2+t^{\frac{2}{5}})^{\frac{1}{2}}.$ In some of the following sections, we adapt the time scale to $t=1$ by a proper scaling, and in this case $\langle x\rangle=(x^2+1)^{\frac{1}{2}}.$ 

We define the spatial Fourier transform of $f\in \mathcal{S}(\mathbb{R})$, the space of all Schwarz functions on $\mathbb{R}$, by \[\hat{f}(\xi)=\frac{1}{\sqrt{2\pi}}\int_{\mathbb{R}}f(x)e^{-ix\xi}dx,\]and its inverse transform by \[{f}^\vee(x)=\frac{1}{\sqrt{2\pi}}\int_{\mathbb{R}}f(\xi)e^{ix\xi}d\xi.\] Both transforms can be extended to $\mathcal{S}'(\mathbb{R})$, the space of tempered distributions, by duality.

Let $u_{kx}(x):=\partial_x^ku(x)=\left((i\xi)^k\hat{u}(\xi)\right)^\vee$ be the $k$-th order partial derivative of $u$ with respect to the spatial variable $x$. Also let $|D|^s$ be the homogeneous fractional derivative of order $s\in\mathbb{R}$ whose symbol is $|\xi|^s$. Note that $|D|^s=H\partial_x$, where $H$ is the Hilbert transform defined by the Fourier symbol $-i\mathrm{sgn}(\xi).$

Let $A,B$ be two nonnegative quantities. If there exists $C>0$ such that $A\leq CB,$ then we denote by $A\lesssim B.$ If $A$ need not be nonnegative and $|A|\lesssim B,$ then we denote $A=O(B).$ If $A\lesssim B$ and $B\lesssim A,$ then we denote $A\sim B.$ If the implied constant $C$ depends on some parameters $a,b,c,\cdots,$ then we denote $A\lesssim_{a,b,c,\cdots}B,$ $A=O_{a,b,c,\cdots}(B),$ $A\sim_{a,b,c,\cdots}B$ respectively.

Let $\Psi\in\mathcal{S}(\mathbb{R})$ with $\psi=\widehat{\Psi}$ is even, supported on $[-2,2]$, and $\psi=1$ on $[-1,1].$ Also let $\psi_{\leq N}(\xi)=\psi(\frac{\xi}{N})$ and $\Psi_{\leq N}=(\psi_{\leq N})^\vee$ for $N>0.$ Then define the Littlewood-Paley projection of $f\in\mathcal{S}'(\mathbb{R})$ by $P_{\leq N}f(x)=(\Psi_{\leq N}\ast f)(x),$ $P_N=P_{\leq N}-P_{\leq N/2}.$ Also define $P_{<N}:=P_{\leq N}-P_N$, $P_{\geq N}=1-P_{<N},$ $P_{>N}=1-P_{\leq N}.$ By a slight abuse of notation, we denote a Fourier multiplier whose symbol is supported on $\{|\xi|\sim N\}$ by $P_N.$ In particular, if $N=2^j\in 2^{\mathbb{Z}},$ then we denote $P_N:=P_j$ and similarly for $P_{\leq N},$ $P_{<N},$ $P_{>N},$ $P_{\geq N}.$

Now define Sobolev and Besov spaces. For $s\in\mathbb{R}$ and $1<p<\infty$, the  inhomogeneous Sobolev space $W^{s,p}$ and its homogeneous counterpart $\dot{W}^{s,p}$ are defined as the subspace of $\mathcal{S}'(\mathbb{R})$ under the norm \[\Vert f\Vert_{W^{s,p}}=\left\Vert \left((1+\xi^2)^{s/2}\hat{f}\right)^\vee\right\Vert_{L^p},\quad \Vert f\Vert_{\dot{W}^{s,p}}=\left\Vert \left(|\xi|^s\hat{f}\right)^\vee\right\Vert_{L^p},\]respectively. Recall the Littlewood-Paley inequality \[\Vert f\Vert_{W^{s,p}}\sim_{s,p}\Vert P_{\leq 0}f\Vert_{L^p}+ \bigg\Vert\big(\sum_{j\geq 1}2^{2js}|P_{j}f |^2\big)^{1/2}\bigg\Vert_{L^p}\] and its homogeneous counterpart \[\Vert f\Vert_{\dot{W}^{s,p}}\sim_{s,p}\bigg\Vert\big(\sum_{j\in\mathbb{Z}}2^{2js}|P_{j}f |^2\big)^{1/2}\bigg\Vert_{L^p}.\]
Motivated by the above Littlewood-Paley inequalities, we define the inhomogeneous and homogeneous Besov spaces, denoted as $B_{s,q}^p$ and $\dot{B}_{s,q}^p$, respectively, is defined as the subspace of $\mathcal{S}'(\mathbb{R})$ under the norm \[\Vert f\Vert_{B_{p,q}^s}=\big(\Vert P_{\leq 1}f\Vert_{L^p}^q+\sum_{j\geq 1}2^{qjs}\Vert P_{j}f\Vert_{L^p}^q\big)^{1/q},\]\[\Vert f\Vert_{\dot{B}_{p,q}^s}=\big(\sum_{j\in\mathbb{Z}}2^{qjs}\Vert P_{j}f\Vert_{L^p}^q\big)^{1/q},\] with $1\leq q<\infty$ and the natural modification for $q=\infty.$
\subsection{Useful theorems}
In this subsection, some of the theorems are presented for the later use. I omit the proofs here, since the proofs can be found in many literatures.
\begin{thm}[Bernstein's inequality] Let $1\leq p\leq q\leq\infty$ and $s\in\mathbb{R}.$ Then one has \begin{align}&\Vert P_Nf\Vert_{L^q}\lesssim_{p,q}N^{\frac{1}{p}-\frac{1}{q}}\Vert P_Nf\Vert_{L^p},\\&\Vert |D|^s P_Nf\Vert_{L^p}\lesssim N^s \Vert P_Nf\Vert_{L^p}.\end{align}
(2.1) still holds when $P_N$ is replaced by $P_{\leq N}.$ Also, if $s> 0,$ then (2.2) holds when $P_N$ is replaced by $P_{\leq N},$ and if $s<0,$ then (2.2) holds when $P_N$ is replaced by $P_{\geq N}.$
\end{thm} In particular, a notable case of Bernstein's inequality is (2.1) with $q=\infty$ and $p=2$, namely $\Vert P_Nf\Vert_{L^\infty}\lesssim N^{\frac{1}{2}}\Vert P_Nf\Vert_{L^2}.$
\\
Also we have the Gagliardo-Nirenberg inequality, a refined version of the Sobolev embedding:

\begin{thm}[Gagliardo-Nirenberg inequality] Let $1<p<q\leq\infty$ and $s>0$ be such that \[\frac{1}{q}=\frac{1}{p}-\frac{\theta s}{d}\] for some $0<\theta<1.$ Then for any $u\in W^{s,p}(\mathbb{R}^d)$ we have \[\Vert u\Vert_{L^q(\mathbb{R}^d)}\lesssim_{d,p,q,s}\Vert u\Vert_{L^p(\mathbb{R}^d)}^{1-\theta}\Vert u\Vert_{\dot{W}^{s,p}(\mathbb{R}^d)}^\theta.\]
\end{thm}
For the proof, see for example \cite{tao2006nonlinear}.

Throughout the paper, a special case of Theorem 2.2, namely \[
    \Vert u\Vert_{L^\infty(\mathbb{R})}\lesssim\Vert u\Vert_{L^2(\mathbb{R})}^{1/2}\Vert u\Vert_{\dot{H}^1(\mathbb{R})}^{1/2},
\]
will be used frequently.

Let $M$ be the one-dimensional Hardy-Littlewood maximal operator defined as \[Mf(x)=\sup_{r>0}\frac{1}{2r}\int_{x-r}^{x+r}|f(y)|dy.\] Then we have following pointwise bound by \cite{fujiwara2000remarks} :\begin{lem}
Let $f$ be a function whose Fourier support is in the annulus $|\xi|\sim R.$ Then for any $x\in\mathbb{R}$ and $s\in\mathbb{R},$\[\left||D|^sf(x)\right|\lesssim_s R^sMf(x).\]Moreover, if $f$ is not necessarily frequency-localized, then\[\left||D|^sP_Nf(x)\right|\lesssim N^s Mf(x).\] In both cases, if $s$ is an integer, then replacing $|D|$ by $\partial_x$ is acceptable. Also, if $s\geq 0$, then the annulus $|\xi|\sim R$ can be replaced by the disk $|\xi|\lesssim R.$
\end{lem}
\begin{rmk}
The above lemma implies that a polynomially decaying pointwise bound is stable under the frequency restriction, since if $g_s(x)=\langle x\rangle^{-s}$ for $s\geq 0,$ then $Mg_s(x)\lesssim_s g_s(x).$ Thus, we can apply the bootstrap bounds of $u$ appearing in the later sections freely to the frequency-localized pieces of $u$.
\end{rmk}

\begin{lem}[Interpolating pointwise bounds] Let $0<s<1$. Then one has \[\left||D|^su(x)\right|\lesssim (Mu(x))^{1-s}(M\partial_xu(x))^s.\]
\end{lem}
\begin{proof}
This is a special case of the theorem in Section 12.3.2 of \cite{mazya2011sobolev}.
\end{proof}
Note that, combining Lemma 2.5 and Remark 2.4, if $|u(x)|\leq A\langle x\rangle^{-\alpha}$ and $|\partial_xu(x)|\leq\langle x\rangle^{-\beta}$ for some $\alpha,\beta\geq 0$ and $A,B>0$, then $||D|^{\frac{1}{2}}u(x)|\lesssim \sqrt{AB}\langle x\rangle^{-\frac{\alpha+\beta}{2}}$. This kind of estimate will be used in the proof of Proposition 4.1.

\section{Linear  Analysis}
First we consider the decay estimate of the solution of linear fifth-order KdV equation \begin{equation}
    u_t-u_{5x}=0.\tag{Lin}\label{eq:Lin}
\end{equation}

We introduce the time-dependent linear operator $L(t):=x+5t\partial_x^4,$ so that if $u$ solves \eqref{eq:Lin} with initial datum $u_0$ then $Lu$ also solves the same equation with initial datum $xu_0.$
\\
Now we state the basic decay bound for the solution of \eqref{eq:Lin}, which is given as below:\begin{prop}\label{prop:lineq} Let $u$ be a solution of \eqref{eq:Lin} with initial datum $u_0\in L^1$. Then $u$ satisfies the uniform decay bound \[\Vert u(t)\Vert_{L^\infty}\lesssim t^{-1/5}\Vert u_0\Vert_{L^1}.\]
\end{prop}
\begin{proof}
Solving the equation via spatial Fourier transform gives \[u(t,x)=\int_{\mathbb{R}}\int_{\mathbb{R}}e^{i\xi^5 t+i\xi(x-y)}u_0(y)dyd\xi.\] We change coordinates $\xi\mapsto t^{-1/5}\eta$ so that \[u(t,x)=t^{-1/5}\int_{\mathbb{R}}\int_{\mathbb{R}}e^{i\eta^5+i\eta t^{-1/5}(x-y)}u_0(y)dyd\eta=t^{-1/5}A(t^{-1/5}x)\ast  u_0(x),\] where $A(x)$ is an oscillatory integral \[A(x):=\int_{\mathbb{R}}e^{i\eta^5+i\eta x}d\eta.\] Then $A$ is bounded by standard stationary phase argument. Now the result follows from Young's inequality.
\end{proof}

To obtain a more refined version of the above estimate, we need the following lemma due to \cite{durugo2014higher} and \cite{fujii2007higher}:
\begin{lem}
 $A(x)$ in the proof of Proposition 3.1 satisfy the following asymptotic bound:\[A(x)\lesssim  \langle x\rangle^{-3/8}e^{-Cx_+^{5/4}},\] where $C$ is an appropriate positive constant.
\end{lem}

From the above lemma, it follows that solutions with initial datum $u_0$ satisfying $\Vert u_0\Vert_{L^1}\leq 1$ and $\mathrm{supp}u_0$ bounded has the following decay bound for $t\gtrsim 1:$\[|u(t,x)|\lesssim t^{-1/8}\langle x\rangle^{-3/8}e^{-Cx_+^{5/4}t^{-1/4}},\quad |u_x(t,x)|\lesssim t^{-3/8}\langle x\rangle^{-1/8}e^{-C'x_+^{5/4}t^{-1/4}}\]
Now we aim to relax the compact support assumption to a decay estimate, in the similar way with \cite{ifrim2019dispersive}. Our goal of this section is to prove the following proposition:
\begin{prop}\label{prop:linbd}
Let $t>0$. Assume that a function $u$ satisfies:\begin{equation}\label{301}\Vert u\Vert_{\dot{B}_{2,\infty}^{-1/2}}+\Vert L(t)u\Vert_{\dot{H}^{1/2}}\leq 1.\end{equation} Then it also satisfies the bound \begin{equation}\label{302}t^{\frac{1}{8}+\frac{k}{4}}\langle x\rangle^{\frac{3}{8}-\frac{k}{4}}|\partial_x^k u|\lesssim 1\end{equation} for $k=0,1,2,3.$ Moreover, in the elliptic region $E=\{x\gg t^{\frac{1}{5}}\}$ we have the better bound\begin{equation}\label{303}t^{\frac{k}{4}}\langle x\rangle^{1-\frac{k}{4}}|\partial_x^ku(x)|\lesssim \log(t^{-\frac{1}{5}}\langle x\rangle).\end{equation}
\end{prop}
This is an analogous result of the Lemma 2.2 in \cite{ifrim2019dispersive}. The only difference with it is the linear operator $L(t),$ where $L(t)=x-3t\partial_x^2$ in \cite{ifrim2019dispersive} and $L(t)=x+5t\partial_x^4$ here.
\
The steps of the proof will basically follow the idea of \cite{ifrim2019dispersive}. First we may rescale to $t=1$ since $\Vert u\Vert_{\dot{B}_{2,\infty}^{-1/2}}+\Vert L(t)u\Vert_{\dot{H}^{1/2}}$ is invariant under the scaling $u(\lambda^5,x)\mapsto \lambda u(\lambda^5,\lambda x).$ Also we will split the real line into the self-similar region $S=\{|x|\lesssim 1\}$ (would be $=\{|x|\lesssim t^{1/5}\}$ in general), the elliptic region $E=\{x\gg 1\}$ (would be $=\{x\gg t^{1/5}\}$ in general), and the hyperbolic region $H=\{-x\gg 1\}.$ Also we split the elliptic and hyperbolic regions into dyadic components, namely $A_R=\{\langle x\rangle\sim R\gg 1\}$ and $A_1=\{\langle x\rangle\lesssim 1\}=S$, $A_R^H=A_R\cap H$, $A_R^E=A_R\cap E$ for $R\gg 1.$ Therefore the proof is reduced to showing that \begin{equation}\label{304}\Vert\partial_x^k u\Vert_{L^\infty(A_R)}\lesssim R^{-\frac{3}{8}+\frac{k}{4}}\end{equation} for each $R\geq 1$ and $k=0,1,2,3$ (and the corresponding counterpart for the elliptic region).
\subsection{Localized bounds of a linear solution}
 Now we present the following low frequency bound, which is the same as in \cite{ifrim2019dispersive}: \begin{lem}\label{lem:lowf1}
If (\ref{301}) holds, we have \begin{equation}\label{305}\Vert Lu\Vert_{L^2(A_R)}\lesssim R^{1/2}.\end{equation}
\end{lem}
\begin{rmk}
The reason why we call \eqref{305} as a low frequency bound is that the suppressed low frequency factors in $\dot{H}^{\frac{1}{2}}$ norm become dominant in the $L^2$ norm. Such low frequency bound turns out to be worse than the high frequency bound, but it is still acceptable thanks to the good low frequency bound of $u$ due to the $\dot{B}^{-\frac{1}{2}}_{2,\infty}$ norm.
\end{rmk}
\begin{proof}
The proof is almost identical to the case of \cite{ifrim2019dispersive}, but I present the details here to clarify the ideas.

First we split $u$ at some frequency cut-off $M$, where $M$ is to be chosen later:\[u=u_{<M}+u_{\geq M}.\] Then the Besov bound from (\ref{301}) tells that $\Vert u_{M}\Vert_{L^2(\mathbb{R})}\lesssim M^{1/2}.$ Hence \[\Vert u_{<M}\Vert_{L^2(\mathbb{R})}\leq\sum_{N<M}\Vert u_N\Vert_{L^2(\mathbb{R})}\lesssim \sum_{N<M}N^{1/2}\sim M^{1/2}\]where the sum runs over all dyadic numbers $<M.$ Therefore one has \begin{equation}\label{306}\Vert Lu_{<M}\Vert_{L^2(A_R)}\leq \Vert xu_{<M}\Vert_{L^2(A_R)}+\Vert \partial_x^4u_{<M}\Vert_{L^2(\mathbb{R})}\lesssim RM^{1/2}+M^{9/2}.\end{equation}

On the other hand, \[Lu_{\geq M}=P_{\geq M}Lu+[L,P_{\geq M}]u=P_{\geq M}Lu+[x,P_{\geq M}]u,\] where the last equality follows since $[\partial_x^4,P_{\geq M}]=0$ (note that any two Fourier multipliers commute). Now we can estimate \[\Vert P_{\geq M}Lu\Vert_{L^2(A_R)}\leq \Vert (1-\psi(\xi/M))\widehat{Lu}\Vert_{L^2(\mathbb{R})}\lesssim M^{-1/2}\Vert Lu\Vert_{\dot{H}^{1/2}(\mathbb{R})}\leq M^{-1/2},\] with $\psi$ being a bump function supported in $|\xi|\lesssim 1.$ Also we have \[\Vert[x,P_{\geq M}]u\Vert_{L^2(A_R)}\leq \Vert M^{-1}\psi'(\xi/M)\hat{u}\Vert_{L^2(\mathbb{R})}\lesssim M^{-1/2}\] by (\ref{301}). This gives \begin{equation}\label{307}\Vert Lu_{\geq M}\Vert_{L^2(A_R)}\lesssim M^{-1/2}.\end{equation}

Now we are going to optimize the bound, compromising between (\ref{306}) and (\ref{307}) by choosing the appropriate $M$. Note that the bound in (\ref{306}) becomes smaller if $M$ is small and the same holds for the bound in (\ref{307}) if $M$ is large. Hence the bound is optimized when the two bounds are comparable. If $RM^{1/2}\sim M^{-1/2},$ then $M\sim R^{-1},$ and the resulting bound becomes $R^{1/2}.$ Also if $M^{9/2}\sim M^{-1/2}$, then $M\sim 1,$ so the resulting bound becomes $R$, and the former is better, so we choose $M=R^{-1}.$ Hence (\ref{305}) follows.
\end{proof}
\begin{rmk} The above argument works in both cases $R=1$ and $R>1.$\end{rmk}
\begin{lem}\label{lem:hif}
If (\ref{301}) holds, we have \begin{equation}\label{308}\Vert \partial_x^ku\Vert_{L^2(A_R)}\lesssim R^{\frac{1}{8}+\frac{k}{4}},\quad k=0,1,2,3,4.\end{equation}
\end{lem}
\begin{rmk}
In \eqref{308}, one more derivative corresponds to the additional $R^{\frac{1}{4}}$ factor. This can be deduced heuristically from the fact that the fifth-order Airy function \[A(t,x)=\int_{\mathbb{R}}e^{itx+i\xi^5t}d\xi\] is concentrated near $(|x|/t)^{\frac{1}{4}}$ with $x<0,$ as the principle of stationary phase implies.
\end{rmk}
\begin{proof}
Again we split $u=u_{<M}+\sum_{\lambda\geq M}u_{\lambda}$ and choose $M$ later as before. The bound on $u_{<M}$ follows by:\begin{equation}\label{309}\Vert \partial_x^ku_{<M}\Vert_{L^2(A_R)}\lesssim M^k\Vert u_{<M}\Vert_{L^2(\mathbb{R})}\lesssim M^{k+\frac{1}{2}}\end{equation} where the last inequality follows from (\ref{301}). Now consider the high frequency bound. The bound of $Lu_{\lambda}$ can be computed in the same way with $Lu_{\geq M}$ in Lemma 3.4, and the bound of $u_\lambda$ directly follows from (\ref{301}) so that \begin{subequations}\begin{equation}\label{310a}\Vert Lu_{\lambda}\Vert_{L^2}\lesssim\lambda^{-1/2},\end{equation} \begin{equation}\label{310b}\Vert u_{\lambda}\Vert_{L^2}\lesssim\lambda^{1/2},\end{equation}\begin{equation}\label{310c} \Vert u_{\lambda,x}\Vert_{L^2}\lesssim \lambda^{3/2}.\end{equation}\end{subequations}Now integrating by parts gives \begin{align*}
    \int \chi_Ru_{\lambda,2x}^2dx=\frac{1}{5}\int \chi_Ru_\lambda Lu_\lambda - \frac{1}{5}\int \chi_Rxu_\lambda^2+\int (\chi_R''u_{\lambda,2x}+2\chi_R'u_{\lambda,3x})u_\lambda.
\end{align*} The first integral can be estimated via Cauchy-Schwarz inequality, which yields the bound 1 by (\ref{310a}) and (\ref{310b}), and the bound of the second integral follows directly from (\ref{310b}), which yields $R\lambda.$ To bound the third integral, we integrate by parts again to obtain:\begin{align*}
\int (\chi_R''u_{\lambda,2x}+2\chi_R'u_{\lambda,3x})u_\lambda=-\frac{1}{2}\int \chi_R^{(4)}u_\lambda^2+2\int \chi_R''u_{\lambda,x}^2.
\end{align*} Hence the third integral is bounded by $R^{-4}\lambda+R^{-2}\lambda^3$ due to (\ref{310b}), (\ref{310c}), and the fact $\Vert \partial_x^k \chi_R\Vert_{L^\infty}=O_{k,\chi}(R^{-k}).$ This gives the preliminary bound \begin{equation}\label{311}\Vert u_{\lambda,2x}\Vert_{L^2(A_R)}\lesssim 1+R^{1/2}\lambda^{1/2}+R^{-1}\lambda^{3/2}:=N.\end{equation}
Repeating the similar steps for the integral $\int \chi_Ru_{\lambda,3x}^2dx$ one gets the similar bound:\begin{equation}\label{311a}
    \Vert u_{\lambda,3x}\Vert_{L^2(A_R)}\lesssim R^{-1}N+\lambda^{1/4}N^{1/2}+R^{1/2}\lambda^{1/4}N^{1/2}:=N'.
\end{equation}
Now we are going to refine (\ref{310b}) and (\ref{310c}) using (\ref{311}) and (\ref{311a}) so that the bound becomes summable with respect to $\lambda.$ To do so we decompose \[\chi_Ru_\lambda=\partial_{x,\lambda}^{-2}(\chi_Ru_{\lambda,2x})-[\partial_{x,\lambda}^{-2},\chi_R]u_{\lambda,2x},\] where the antiderivative $\partial_{x,\lambda}^{-2}$ is localized at frequency $\lambda,$ namely having the symbol $-\xi^{-2}\varphi(\xi/\lambda)$ with $\varphi$ being a bump function supported on the annulus $|\xi|\sim 1.$ Let $\psi$ be the Fourier inversion of $-\xi^{-2}\varphi(\xi),$ so that $\psi\in\mathcal{S}(\mathbb{R}).$ Then the inversion of $-\xi^{-2}\varphi(\xi/\lambda)$ becomes $\lambda^{-1}\psi(\lambda x).$ Then one has \begin{align*}
\Vert [\partial_{x,\lambda}^{-2},\chi_R]u_{\lambda,2x}\Vert_{L^2(\mathbb{R})}&=\lambda^{-1}\left\Vert\int \psi(\lambda(x-y))(\chi_R(y)-\chi_R(x))u_{\lambda,2x}(y)dy\right\Vert_{L^2_x(\mathbb{R})}\\&\lesssim\lambda^{-1}\Vert |\psi(\lambda\cdot)|\ast |u_{\lambda,2x}|\Vert_{L^2_x}\leq \lambda^{-2}\Vert \psi\Vert_{L_x^1}\Vert u_{\lambda,2x}\Vert_{L_x^2}\\&\lesssim \lambda^{-2}N.
\end{align*} Also \[\Vert \partial_{x,\lambda}^{-2}(\chi_Ru_{\lambda,2x})\Vert_{L^2(\mathbb{R})}\lesssim \lambda^{-2}\Vert u_{\lambda,2x}\Vert_{L^2(A_R)}\lesssim \lambda^{-2}N.\] This yields a local bound for $u_\lambda$:\begin{equation}\label{313}\Vert u_\lambda\Vert_{L^2(A_R)}\lesssim \lambda^{-2}N.\end{equation}
Now summing up (\ref{313}) on the frequency range $\lambda\geq M$ gives\begin{equation}\label{314}
\Vert u_{\geq M}\Vert_{L^2(A_R)}\lesssim
M^{-2}+R^{1/2}M^{-3/2}+R^{-1}M^{-1/2}.\end{equation}
Repeating the same argument above to the decomposition \begin{equation}\label{315}\chi_Ru_{\lambda,x}=\partial^{-2}_{x,\lambda}(\chi_Ru_{\lambda,3x})-[\partial^{-2}_{x,\lambda},\chi_R]u_{\lambda,3x}\end{equation} we also get the refined local bound for $u_{\lambda,x}$:\begin{equation}\label{316}\Vert u_{\lambda,x}\Vert_{L^2(A_R)}\lesssim \lambda^{-2}N',\end{equation} so that \begin{equation}\label{317}
    \Vert \partial_xu_{\geq M}\Vert_{L^2(A_R)}\lesssim R^{-1}M^{-2}+R^{-2}M^{-\frac{1}{2}}+R^{\frac{1}{2}}M^{-\frac{7}{4}}+R^{\frac{3}{4}}M^{-\frac{3}{2}}+M^{-1}.
\end{equation}
Now we want to optimize the bounds (\ref{309}), (\ref{314}), (\ref{317}) by making them comparable. It turns out that $M=R^{1/4}$ is the case with the best bound, so that $M^{-2}N\sim R^{1/8}.$ This gives \begin{subequations}\begin{equation}\label{318a}\Vert u\Vert_{L^2(A_R)}\lesssim R^{1/8},\end{equation}\begin{equation}\label{318b}\quad \Vert u_{x}\Vert_{L^2(A_R)}\lesssim R^{3/8},\end{equation}\begin{equation}\label{318c}\quad \Vert u_{2x}\Vert_{L^2(A_R)}\lesssim R^{5/8}.\end{equation}\end{subequations}
Here the bound (\ref{318c}) follows from integration by parts below with (\ref{305}), (\ref{318a}), and (\ref{318b}):
\begin{align*}
    \int \chi_Ru_{2x}^2dx=\frac{1}{5}\int \chi_Ru Lu - \frac{1}{5}\int \chi_Rxu^2+\int (\chi_R''u_{2x}+2\chi_R'u_{3x})u.
\end{align*}
Using the bound (\ref{305}) and (\ref{318a}) easily gives that \begin{equation}\label{319}\Vert u_{4x}\Vert_{L^2(A_R)}\lesssim R^{9/8}.\end{equation} Finally, it remains to find the bound of $\Vert u_{3x}\Vert_{L^2(A_R)}.$ This can be done easily by integration by parts:\[\int \chi_Ru_{3x}^2=-\int\chi_R'u_{3x}u_{2x}-\int\chi_Ru_{4x}u_{2x}=\frac{1}{2}\int\chi_R''u_{2x}^2-\int\chi_Ru_{4x}u_{2x}\]so the bound follows from (\ref{318c}) and (\ref{319}).
\end{proof}
\subsection{Pointwise estimation on each region}\label{sec:302}
Let $v:=\chi_Ru.$ For notational convenience, we are going to restrict the norm to the region $A_R,$ so we will write $L^2(A_R)$ simply as $L^2.$ Then $v$ solves an equation of the form \[(x+5\partial_x^4)v=f\] in $A_R,$ where we control \begin{equation}\label{320}\Vert \partial_x^kv\Vert_{L^2}\lesssim R^{\frac{1}{8}+\frac{k}{4}},\quad 0\leq k\leq 4\end{equation} by (\ref{308}), and \begin{equation}\label{321}\Vert f\Vert_{\dot{H}^{1/2}}\lesssim 1,\quad \Vert f\Vert_{L^2}\lesssim R^{\frac{1}{2}}\end{equation} which follows by (\ref{301}), (\ref{305}).

Now we turn to the pointwise estimate in each region. In the case of self-similar region, the result follows directly from Sobolev embedding (note that in this region just $\lesssim 1$ bound is enough, since $R=1$). Hence we concentrate on elliptic and hyperbolic regions.

First we see the elliptic region, namely $x\sim R,$ which seems to be easier to deal with than the hyperbolic region. Split $f=f_{lo}+f_{hi},$ where the frequency scale is $R^{1/4}.$ Here the operator $L=x+5\partial_x^4$ is elliptic, so the leading part will then be $x^{-1}f_{lo}.$ Then let $v_1:=v-x^{-1}f_{lo}$ be the remainder, which solves \[Lv_1=f_1:=f_{hi}-5\partial_x^4(x^{-1}f_{lo}).\]
By (2) one has \[\Vert f_{hi}\Vert_{L^2(\mathbb{R})}\lesssim R^{-1/8}\Vert f\Vert_{\dot{H}^{1/2}}\lesssim R^{-1/8},\]\begin{align*}\Vert \partial_x^4 (x^{-1}f_{lo})\Vert_{L^2(\mathbb{R})}\lesssim &R^{-5}\Vert f_{lo}\Vert_{L^2(\mathbb{R})}+R^{-4}\Vert \partial_xf_{lo}\Vert_{L^2(\mathbb{R})}+R^{-3}\Vert \partial_x^2f_{lo}\Vert_{L^2(\mathbb{R})}\\+&R^{-2}\Vert \partial_x^3f_{lo}\Vert_{L^2(\mathbb{R})}+R^{-1}\Vert \partial_x^4f_{lo}\Vert_{L^2(\mathbb{R})}\\\lesssim& R^{-9/2}+R^{-31/8}+R^{-21/8}+R^{-11/8}+R^{-1/8}\\\lesssim &R^{-1/8},\end{align*} where for $k\geq 1$, \[\Vert \partial_x^k f_{lo}\Vert_{L^2(\mathbb{R})}\lesssim R^{\frac{k}{4}-\frac{1}{8}}\Vert f_{lo}\Vert_{\dot{H}^{1/2}}\lesssim R^{\frac{k}{4}-\frac{1}{8}},\] which yield \[\Vert f_1\Vert_{L^2(\mathbb{R})}\lesssim R^{\frac{k}{4}-\frac{1}{8}}.\]

Hence we may integrate by parts in the following identity \[\int_{\mathbb{R}}v_1Lv_1dx=\int_{\mathbb{R}}f_1v_1dx\] and arrive at \[\int_{\mathbb{R}}x |v_1|^2 dx +5\int_{\mathbb{R}}|\partial_x^2v_1|^2dx=\int_{\mathbb{R}}f_1v_1dx.\] 
Then by Cauchy-Schwarz inequality,\[R\Vert v_1\Vert_{L^2(\mathbb{R})}^2+5\Vert \partial_x^2v_1\Vert_{L^2(\mathbb{R})}^2\lesssim \Vert f_1\Vert_{L^2(\mathbb{R})}\Vert v_1\Vert_{L^2(\mathbb{R})}\leq \delta R\Vert v_1\Vert_{L^2(\mathbb{R})}^2+\frac{4}{\delta R}\Vert f_1\Vert_{L^2(\mathbb{R})}^2\] for any $\delta>0,$ so letting $\delta>0$ sufficiently small gives \[R\Vert v_1\Vert_{L^2(\mathbb{R})}^2+5\Vert \partial_x^2v_1\Vert_{L^2(\mathbb{R})}^2\lesssim R^{-1}\Vert f_1\Vert_{L^2(\mathbb{R})}^2.\] Thus, using the bound on $f_1,$ we arrive at \[\Vert \partial_x^kv_1\Vert_{L^2(\mathbb{R})}\lesssim R^{-\frac{9}{8}+\frac{k}{4}},\quad k=0,2,\] and using the $v_1$ equation shows that this bound also holds for $k=4,$ and by interpolation it still holds for $k=1,3.$ To sum up, we have\[\Vert \partial_x^kv_1\Vert_{L^2(\mathbb{R})}\lesssim R^{-\frac{9}{8}+\frac{k}{4}},\quad k=0,1,2,3,4.\]
Finally, by Gagliardo-Nirenberg inequality, one has \[|\partial_x^k v_1|\lesssim R^{-1+\frac{k}{4}},\quad k=0,1,2,3,\] which is exactly as needed.

Now we need to get the bound of $x^{-1}f_{lo}.$ Proceeding as above, one has \[\Vert \partial_x^k(x^{-1}f_{lo})\Vert_{L^2(\mathbb{R})} \lesssim R^{-\frac{9}{8}+\frac{k}{4}}\] for $k=1,2,3,4.$ Hence again Gagliardo-Nirenberg inequality gives \[\Vert x^{-1}f_{lo}\Vert_{L^2(\mathbb{R})}\lesssim R^{-1+\frac{k}{4}},\quad k=1,2,3.\] For $k=0$, we decompose $f_{lo}$ further as \[f_{lo}=f_{<R^{-1}}+\sum_{R^{-1}\leq \lambda \leq R^{1/4}}f_\lambda.\] Then Bernstein's inequality and Littlewood-Paley characterization of Sobolev norms give \begin{align*}
    \Vert f_{<R^{-1}}\Vert_{L^\infty(\mathbb{R})}&\lesssim R^{-1/2}\Vert f_{<R^{-1}}\Vert_{L^2(\mathbb{R})}\lesssim 1,\\\sum_{R^{-1}\leq \lambda \leq R^{1/4}}\Vert f_{\lambda}\Vert_{L^\infty(\mathbb{R})}&\lesssim \sum_{R^{-1}\leq \lambda \leq R^{1/4}}\lambda^{1/2}\Vert f_{\lambda}\Vert_{L^2(\mathbb{R})}\\&\lesssim (\log R)^{1/2}\Big(\sum_{R^{-1}\leq \lambda \leq R^{1/4}}\lambda\Vert f_{\lambda}\Vert_{L^2(\mathbb{R})}^2 \Big)^{1/2}\\&\lesssim (\log R)^{1/2}\Vert f\Vert_{\dot{H}^{1/2}}\lesssim (\log R)^{1/2}.
\end{align*}
where the inequality used in the penultimate line is just an elementary inequality \[\sum_{k=1}^n a_k\leq (n\sum_{k=1}^n a_k^2)^{1/2}.\] Hence one has \begin{equation}\label{322}\Vert x^{-1}f_{lo}\Vert_{L^\infty}\lesssim R^{-1}(\log R)^{1/2},\end{equation} which is exactly needed.

Finally we turn to the hyperbolic region. As in the third-order case, $f$ might be worse at lower frequency (below $R^{\frac{1}{4}}$) than higher frequencies, due to the weight $|\xi|^{1/2}.$ Hence we proceed by splitting $f$ to low and high frequencies:\[f=\widetilde{\chi}_Rf_{\leq R^{1/4}}+\widetilde{\chi}_Rf_{> R^{1/4}}:=f_{lo}+f_{hi},\] where $\tilde{\chi}_R$ is smooth, supported on $|x|\sim R,$ and $\tilde{\chi}_R\chi_R=\chi_R$.

Then we may use an energy estimate,\[\frac{d}{dx}(-x|v|^2+5|v_{xx}|^2-10v_xv_{3x}+2f_{lo}v)=-|v|^2-2f_{hi}v_x+2f_{lo,x}v\leq -2fv_x+2f_{lo,x},\] and integrating both sides gives \begin{equation}\label{323'}-x|v|^2+5|v_{xx}|^2-10v_xv_{3x}\lesssim \Vert f_{hi}\Vert_{L^2}\Vert v_x\Vert_{L^2}+\Vert f_{lo,x}\Vert_{L^2}\Vert v\Vert_{L^2}+|f_{lo}v|\end{equation} whenever $x\in A_R^H$, so the term  $|f_{lo}v|$ can be controlled via (\ref{322}) and the Gagliardo-Nirenberg inequality \[\Vert v\Vert_{L^\infty}\lesssim \Vert v\Vert_{L^2}^{\frac{1}{2}}\Vert v_x\Vert_{L^2}^{\frac{1}{2}}\lesssim R^{\frac{1}{4}},\] and the other terms can be bounded by (\ref{320}) and (\ref{321}). This finally gives \begin{equation}\label{323}-x|v|^2+5|v_{xx}|^2-10v_xv_{3x}\lesssim R^{\frac{1}{4}}(\log R)^{\frac{1}{2}}.\end{equation}Now we have a task to deal with the extra $v_xv_{3x}$ term, which was absent in the third-order case. The problem is that, we should control the pointwise bound of $v_xv_{3x}$, either.

Now try a similar estimate as follows:\begin{align*}
    \frac{d}{dx}&(-x|v_x|^2+5|v_{3x}|^2+2xvv_{xx}+2f_{lo}v_{2x})\leq 2vv_{xx}+2f_{hi}v_{3x}-2f_{lo,x}v_{2x}\\\lesssim &|x|^{-\frac{1}{2}}(-x|v|^2+5|v_{xx}|^2)+|f_{hi}v_{3x}|+|f_{lo,x}v_{2x}|\\\lesssim& R^{-\frac{1}{2}}(R^{\frac{1}{4}}(\log R)^{\frac{1}{2}}+10v_xv_{3x})+|f_{hi}v_{3x}|+|f_{lo,x}v_{2x}|,
\end{align*} and integrating both sides, together with Cauchy-Schwarz inequality gives \begin{align*}
&-x|v_x|^2+5|v_{3x}|^2+2xvv_{xx}\\\lesssim& R^{\frac{3}{4}}(\log R)^{\frac{1}{2}}+R^{-\frac{1}{2}}\Vert v_x\Vert_{L^2}\cdot \Vert v_{3x}\Vert_{L^2}+\Vert f_{hi}\Vert_{L^2}\Vert v_{3x}\Vert_{L^2}+\Vert f_{lo,x}\Vert_{L^2}\Vert v_{2x}\Vert_{L^2}+|f_{lo}v_{2x}|\\\lesssim& R^{\frac{3}{4}}(\log R)^{\frac{1}{2}}+|f_{lo}v_{2x}|
\end{align*} for $x\in A_R^H$. Here the term $|f_{lo}v_{2x}|$ can be treated in the same way as above, so one has \begin{equation}\label{324}-x|v_x|^2+5|v_{3x}|^2+2xvv_{xx}\lesssim R^{\frac{3}{4}}(\log R)^{\frac{1}{2}}.\end{equation}
Now we obtained the system of two inequalities:\begin{subequations}
    \begin{equation}\label{325a}-x|v|^2+5|v_{xx}|^2-10v_xv_{3x}
    \leq CR^{\frac{1}{4}}(\log R)^{\frac{1}{2}},\end{equation}\begin{equation}\label{325b}-x|v_x|^2+5|v_{3x}|^2+2xvv_{xx}\leq CR^{\frac{3}{4}}(\log R)^{\frac{1}{2}}.\end{equation}
\end{subequations} From (\ref{325a}) and (\ref{325b}) we have \begin{align*}
 -x|v|^2+5|v_{xx}|^2 \leq& CR^{\frac{1}{4}}(\log R)^{\frac{1}{2}}+10v_xv_{3x}\leq CR^{\frac{1}{4}}(\log R)^{\frac{1}{2}}+\sqrt{\frac{5}{-x}}(-xv_x^2+5v_{3x}^2)\\\leq&CR^{\frac{1}{4}}(\log R)^{\frac{1}{2}}+2\sqrt{-5x}vv_{2x},
    \end{align*} so that \begin{equation}\label{326a}|\sqrt{-x}v-\sqrt{5}v_{2x}|\lesssim R^{\frac{1}{8}}(\log R)^{\frac{1}{4}}\end{equation} and in the same way one can show \begin{equation}\label{326b}|\sqrt{-x}v_x-\sqrt{5}v_{3x}|\lesssim R^{\frac{3}{8}}(\log R)^{\frac{1}{4}}.\end{equation}
Now it remains to show that \begin{equation}\label{lastgoal}\Vert \partial_x^kv\Vert_{L^\infty(A_R)}\lesssim R^{-\frac{3}{8}+\frac{k}{4}}(\log R)^{\frac{1}{4}},\quad k=0,1,2,3.\end{equation} To deduce a contradiction, suppose that $\Vert v\Vert_{L^\infty(A_R)}\not\lesssim R^{-\frac{3}{8}}(\log R)^{\frac{1}{4}}$. Then there exists a sequence $\{R_n\}\uparrow\infty$ as $n\to\infty$ such that $\Vert v\Vert_{L^\infty(A_{R_n})}\geq 2^n R_n^{-\frac{3}{8}}(\log R_n)^{\frac{1}{4}}.$ Then by \eqref{326a} it turns out that $\Vert v_{2x}\Vert_{L^\infty(A_{R_n})}\gtrsim 2^n R_n^{\frac{1}{8}}(\log R_n)^{\frac{1}{4}}$ whenever $n$ is large enough. Also by \eqref{325a} and \eqref{326b} one may show that $\Vert v_{x}\Vert_{L^\infty(A_{R_n})}\gtrsim 2^n R_n^{-\frac{1}{8}}(\log R_n)^{\frac{1}{4}}$ and $\Vert v_{3x}\Vert_{L^\infty(A_{R_n})}\gtrsim 2^n R_n^{\frac{3}{8}}(\log R_n)^{\frac{1}{4}}$ again for $n$ large enough. Now define the set $B_n:=\{x:|v(x)|\geq 2^{n-1}R_n^{-\frac{3}{8}}(\log R_n)^{\frac{1}{4}}\}.$ Here $v$ is assumed to be differentiable sufficiently many times, so by taking a connected component we may assume that $B_n=[a_n,b_n]$, where $a_n$ and $b_n$ are the first left and right points where $v(a_n)=v(b_n)=2^{n-1}R_n^{-\frac{3}{8}}(\log R_n)^{\frac{1}{4}}.$ Note that $-\infty<a_n$ and $b_n<\infty$ because $v$ is compactly supported.

Now observe by (\ref{325a}) one has
\begin{equation}\label{327}
    v_xv_{3x}\geq \frac{1}{10}\left(-x|v|^2+5|v_{2x}|^2-CR_n^{\frac{1}{4}}(\log R_n)^{\frac{1}{2}}\right)\gtrsim 2^{2(n-1)} R_n^{\frac{1}{4}}(\log R_n)^{\frac{1}{2}}
\end{equation} whenever $x\in B_{n}$ with $n$ large enough.
Similarly, by (\ref{325b}) one also has \begin{equation}\label{326}
\begin{split}
    vv_{2x}&\geq \frac{1}{-2x}\left(-x|v_x|^2+5|v_{3x}|^2-CR_n^{\frac{3}{4}}(\log R_n)^{\frac{1}{2}}\right)\\&\geq \frac{1}{-2x}\left(2\sqrt{-5x}v_xv_{3x}-CR_n^{\frac{3}{4}}(\log R_n)^{\frac{1}{2}}\right)\gtrsim 2^{2(n-1)}R_n^{-\frac{1}{4}}(\log R_n)^{\frac{1}{2}}
\end{split}
\end{equation} under the same assumption above. (\ref{326}) shows that for sufficiently large $n$, $v$ and $v_{2x}$ always have the same sign on $B_n$, and the similar statement holds for $v_x$ and $v_{3x}.$

First we may assume that $v>0$ on $B_n$ (otherwise we may replace $v$ by $-v$). Then $v\geq 2^{n-1}R_n^{-\frac{3}{8}}(\log R_n)^{\frac{1}{4}}$ and it is strictly convex on $B_n$ since $v_{2x}>0$. Also by definition $v(a_n)=v(b_n)=2^{n-1}R_n^{-\frac{3}{8}}(\log R_n)^{\frac{1}{4}}$. However, this is a contradiction because strict convexity means $v(ta_n+(1-t)b_n)<tv(a_n)+(1-t)v(b_n)=2^{n-1}R_n^{-\frac{3}{8}}(\log R_n)^{\frac{1}{4}}$ on $B_n.$  Hence it follows that $\Vert v\Vert_{L^\infty}\lesssim R^{-\frac{3}{8}}(\log R)^{\frac{1}{4}}$. This proves (\ref{lastgoal}), and plugging (\ref{lastgoal}) into the right hand side of the inequality (\ref{323'}) gives \begin{subequations}
    \begin{equation}-x|v|^2+5|v_{xx}|^2-10v_xv_{3x}
    \lesssim R^{\frac{1}{4}},\end{equation}\begin{equation}-x|v_x|^2+5|v_{3x}|^2+2xvv_{xx}\lesssim R^{\frac{3}{4}}.\end{equation}
\end{subequations} and repeating the same argument removes the logarithm from \eqref{lastgoal}, and the Proposition 3.3 is proved.
\begin{rmk}\begin{enumerate}[(1)]
\item The overview of the linear analysis is similar to that of \cite{ifrim2019dispersive}. However, since we have to deal with the higher order equation, there are some additional steps which was not in \cite{ifrim2019dispersive}. More specifically, in \eqref{325a} and \eqref{325b} we had to control the cross terms which did not appear in \cite{ifrim2019dispersive}. This is one of the main obstacles we have to pay attention, since such cross terms may become more and more complicated as the order of the linear equation becomes higher, making the linear analysis more and more difficult.
\item A notable advantage of the argument is that it does not depend on the fact that $u$ is a solution of \eqref{eq:Lin}. Thus, as one can see in the last section, the same argument can be applied to the nonlinear case after a slight modification.
\end{enumerate}
\end{rmk}

\section{Outline of the nonlinear analysis}
Let $u$ be a solution to \eqref{eq:KW} with the smallness assumption of the initial datum \begin{equation}\label{small}\Vert u_0\Vert_{\dot{B}_{2,\infty}^{-\frac{1}{2}}}+\Vert xu_0\Vert_{\dot{H}^{\frac{1}{2}}}\leq \epsilon\ll 1.\end{equation} Motivated from Proposition 3.3, we are going to make bootstrap assumptions \begin{equation}\label{bootstrap}
    |u_{kx}(x,t)|\leq M\epsilon t^{-\frac{1}{8}-\frac{k}{4}}\langle x\rangle^{\frac{k}{4}-\frac{3}{8}},\quad k=0,1,2,3,
\end{equation} where $M$ is a large universal constant, independent of $\epsilon,$ to be chosen later, and $\epsilon\ll 1$ is small depending on $M$. 
The overall process for nonlinear estimates is similar with that in \cite{ifrim2019dispersive}, sometimes being simpler. More precisely, the process may include energy estimates for both nonlinear and linearized equations, uniform Sobolev bounds for $L^{NL}u$ for some nonlinear operator $L^{NL},$ and finally the nonlinear dispersive bounds.

Let us explain the overall strategies of this paper. Mainly we focus on the case of \eqref{eq:KW}, since the case of \eqref{eq:mKW} is entirely similar. The case of \eqref{eq:mKW} is briefly discussed in the appendix.
\\
Basically we are going to follow the argument of \cite{ifrim2019dispersive}, but the case becomes much simpler, thanks to the higher dispersive effect on high frequencies.
\\
To prove our main result, Theorem 1.1, we are going to split the proof into few steps. Below is the list of the propositions to be proved.
\begin{prop}
     Let $u$ be a solution to \eqref{eq:KW} with the smallness assumption of the initial datum \eqref{small} together with the bootstrap bound \eqref{bootstrap}. Then one has \[\Vert u(t)\Vert_{\dot{B}_{2,\infty}^{-\frac{1}{2}}}\lesssim\epsilon\] for the time scale $|t|\ll_M\epsilon^{-\frac{5}{3}}.$
\end{prop}
Since the proof of Proposition 4.1 is very short, we give the proof right here.\\A direct energy estimate gives
\begin{align*}
\frac{d}{dt}\int (P_\lambda  u)^2 =\frac{1}{2}\int P_\lambda u\cdot\partial_xP_{\lambda}(u^2)\lesssim \Vert P_\lambda u\Vert_{L^2}\lambda^{\frac{1}{2}}\Vert |D|^{\frac{1}{2}}(u^2)\Vert_{L^2}.
\end{align*} Also we may find the bound of $\Vert |D|^{\frac{1}{2}}(u^2)\Vert_{L^2}$ as follows: From $|u^2|\leq M^2\epsilon^2 t^{-\frac{1}{4}}\langle x\rangle^{-\frac{3}{4}}$, $|\partial_x(u^2)|\leq M^2\epsilon^2 t^{-\frac{1}{2}}\langle x\rangle^{-\frac{1}{2}}$, and Lemma 2.5 one has $||D|^{\frac{1}{2}}(u^2)|\lesssim M^2\epsilon^2 t^{-\frac{3}{8}} \langle x\rangle^{-\frac{5}{8}}$, so that \begin{align*}\Vert |D|^{\frac{1}{2}}(u^2)\Vert_{L^2}\lesssim &M^2\epsilon^2 t^{-\frac{3}{8}}\left(\int (x^2+t^{\frac{2}{5}})^{-\frac{5}{8}}dx\right)^{\frac{1}{2}}= M^2\epsilon^2 t^{-\frac{3}{8}}\left(\int (y^2+1)^{-\frac{5}{4}}t^{-\frac{1}{4}}t^{\frac{1}{5}}dy\right)^{\frac{1}{2}}\\\lesssim& M^2\epsilon^2 t^{-\frac{2}{5}}\end{align*} where the change of variables $x:=t^{\frac{1}{5}}y$ is used. Hence it should be the case that \[\Vert P_\lambda u\Vert_{L^2}\lesssim \Vert P_\lambda u_0\Vert_{L^2}+\lambda^{\frac{1}{2}}M^2\epsilon^2t^{\frac{3}{5}}\lesssim \epsilon\lambda^{\frac{1}{2}}\] under the assumption of the time bound $M^2\epsilon t^{\frac{3}{5}}\leq 1$.\footnote{This is why the time scale for \eqref{eq:KW} in Theorem 1.1 is restricted to $|t|\ll \epsilon^{-\frac{5}{3}}$.} Therefore, we get the desired bound \[\Vert u(t)\Vert_{\dot{B}^{-\frac{1}{2}}_{2,\infty}}\lesssim\epsilon,\quad t\ll_M\epsilon^{-\frac{5}{3}}.\]
\begin{rmk} The reason why the direct energy estimate of the Besov bound for \eqref{eq:KW} is possible is that it has a better spatial decay of the solution that it is directly square integrable. One may see that the above argument does not work for \eqref{eq:5KdV} unless $c_1=c_2=0$\footnote{The case $c_1=c_2=0$ is merely \eqref{eq:mKW}, where we can apply the same argument as discussed in the appendix.} due to more derivatives in the nonlinearity and thus having worse spatial decay. It is also not applicable to the KdV equation for the same reason, which is why Ifrim et al. \cite{ifrim2019dispersive} used the small data $H^{-1}$ conservation law of KdV equation, which is aided from the complete integrability.\end{rmk}
\begin{prop} Let $u$ be a solution to \eqref{eq:KW} which satisfies the smallness assumption \eqref{small} for the initial data, as well as the bootstrap assumption \eqref{bootstrap}. Then the linearized equation  \begin{equation}
    z_t-z_{5x}=uz_x\tag{LinKW}\label{eq:LinKW}
\end{equation} is well-posed in $\dot{H}^{\frac{1}{2}}$ with uniform bounds \[\Vert z(t)\Vert_{\dot{H}^{\frac{1}{2}}}\sim \Vert z(0)\Vert_{\dot{H}^{\frac{1}{2}}},\quad t\ll_M\epsilon^{-\frac{5}{3}}.\]\end{prop}
The proof of Proposition 4.3 will be discussed in the Section 5. Using Proposition 4.3, one may prove the following proposition:
\begin{prop}
Let $u$ be a solution to \eqref{eq:KW} satisfying the smallness assumption \eqref{small} and the bootstrap assumption \eqref{bootstrap}. Then $L^{NL}u:=xu+5tu_{4x}+\frac{5}{2}tu^2$ satisfies the bound \[\Vert L^{NL}u(t)\Vert_{\dot{H}^{\frac{1}{2}}}\lesssim \epsilon,\quad |t|\ll_M \epsilon^{-\frac{5}{3}}.\]
\end{prop} Proposition 4.4 will be proved in the Section 6.\\
Let me explain the relations between the above propositions and our main theorem. First, gathering the results from the above propositions, we conclude the following: Under the smallness assumption \eqref{small}, we have the bound \begin{equation}\label{NL}\Vert u(t)\Vert_{\dot{B}_{2,\infty}^{-\frac{1}{2}}}+\Vert L^{NL}u(t)\Vert_{\dot{H}^{\frac{1}{2}}}\lesssim \epsilon,\quad |t|\ll_M\epsilon^{-\frac{5}{3}}.\end{equation} Under the bound \eqref{NL}, we may apply the same argument as in the Section 3, with some additional procedures to control the nonlinear terms. This will be done in the following subsection. Finally, in the Appendix, we briefly sketch the proof of our main result in the case of \eqref{eq:mKW}.

\subsection{Proof of Theorem 1.1}
In this section, we finish the proof of our main theorem, Theorem 1.1, assuming that the Propositions 4.1, 4.3, and 4.4 are true.
For the sake of simplicity, we are going to rescale the problem to $t=1.$ Namely, given the equation \[xu+5tu_{4x}+\frac{5}{2}tu^2=f,\quad f=L^{NL}u\]let $\tilde{u}(x)=t^{\frac{4}{5}}u(t,xt^{\frac{1}{5}}),\quad \tilde{f}(x)=t^{\frac{3}{5}}f(t,xt^{\frac{1}{5}}).$ Then a direct calculation gives that $\tilde{u}$ and $\tilde{f}$ solve the same equation with $t=1,$ so that \begin{equation}\label{701}(x+5\partial_x^4)\tilde{u}+\frac{5}{2}\tilde{u}^2=\tilde{f}.\end{equation} Such scaling also gives the new bounds as follows:\begin{equation}\label{702}\Vert \tilde{u}\Vert_{\dot{B}_{2,\infty}^{-\frac{1}{2}}}\lesssim\tilde{\epsilon}:=\epsilon t^{\frac{3}{5}}\ll 1,\quad |\partial_x^k\tilde{u}(x)|\lesssim M\tilde{\epsilon}\langle x\rangle^{-\frac{3}{8}+\frac{k}{4}},\;k=0,1,2,3,\end{equation} where $\langle x\rangle=(x^2+1)^{\frac{1}{2}}$ does not depend on $t$ and \begin{equation}\label{703}\Vert \tilde{f}\Vert_{\dot{H}^{\frac{1}{2}}}\lesssim \tilde{\epsilon},\end{equation}
and moreover, the rescaled time bound assumption \begin{equation}\label{rescaled} M^2\tilde{\epsilon}\leq 1.\end{equation}
For notational convenience, we are going to drop the tilde notation. Then the proof is reduced to showing that \[\Vert \partial_x^ku\Vert_{L^\infty(A_R)}\lesssim \epsilon R^{-\frac{3}{8}+\frac{k}{4}}\] and the corresponding counterpart for the elliptic region $A_R\cap E.$
\\
Now most of the procedures are the same with the linear analysis, so I am going to skip most of the details, except the necessary nonlinear analysis.

\begin{lem}
Under the assumptions of (\ref{701}), (\ref{702}), (\ref{703}), and (\ref{rescaled}), we have \begin{equation}\label{704} \Vert f\Vert_{L^2(A_R)}\lesssim \epsilon R^{\frac{1}{2}}.
\end{equation}
\end{lem}
\begin{proof}
As in the proof of Lemma 3.4, we split $u$ at the frequency cutoff $R^{-1}$ and compute:
\[f=P_{>R^{-1}}f+P_{<R^{-1}}Lu+P_{<R^{-1}}(u^2).\footnote{Here I intentionally dropped any insignificant constant coefficients.}\]The first two terms can be estimated in the same way as Lemma 3.4. For the last term, the pointwise bootstrap bound \eqref{702} and Lemma 2.3 gives \[|P_{<R^{-1}}(u^2)|\lesssim  M^2\epsilon^2 R^{-\frac{3}{4}},\]so that \[\Vert P_{<R^{-1}}(u^2)\Vert_{L^2(A_R)}\lesssim M^2\epsilon^2 R^{-\frac{1}{4}},\] and using the time bound \eqref{rescaled}, the desired bound \eqref{704} is obtained.
\end{proof}

\begin{lem}
Under the assumptions of (\ref{701}), (\ref{702}), (\ref{703}), and (\ref{rescaled}), we have \begin{equation}\label{705}
\Vert \partial_x^ku\Vert_{L^2(A_R)}\lesssim \epsilon R^{\frac{1}{8}+\frac{k}{4}}.
\end{equation}
\end{lem}
\begin{proof}
The low frequency case $\lambda\leq R^{\frac{1}{4}}$ can be dealt with in the same way in the linear case, so we focus on the high frequency case, $\lambda> R^{\frac{1}{4}}$. Again we split as \[Lu_\lambda=P_\lambda L^{NL}u+\lambda^{-1}u_{\lambda}-P_\lambda(u^2),\] and the first two terms can be estimated in the same way as in the linear case, so we again focus on the nonlinear terms. Using (\ref{702}) and Lemma 2.3 gives \[\Vert P_\lambda(u^2)\Vert_{L^2(A_R)}\lesssim \epsilon R^{-\frac{1}{4}}\] and \[\Vert \partial_xP_\lambda(u^2)\Vert_{L^2(A_R)}\lesssim \epsilon R^{\frac{1}{4}},\] so that \[\Vert P_\lambda(u^2)\Vert_{L^2(A_R)}\lesssim \epsilon \lambda^{-\frac{1}{2}}.\]Now the argument is completed as in the linear case.
\end{proof}

Now we may localize the solution $u$ to the dyadic regions $|x|\sim R.$ Let $v:=\chi_Ru,$ then $v$ solves the equation \begin{equation}\label{706}(x+5\partial_x^4)v+\frac{5}{2}uv=g,\end{equation}with the bounds
\begin{equation}\label{707}
\Vert \partial_x^k v\Vert_{L^2(A_R)}\lesssim \epsilon R^{\frac{1}{8}+\frac{k}{4}},\quad k=0,1,2,3,4,
\end{equation}and \begin{equation}\label{708}
\Vert g\Vert_{\dot{H}^{\frac{1}{2}}}\lesssim \epsilon,\quad \Vert g\Vert_{L^2(A_R)}\lesssim \epsilon R^{\frac{1}{2}}.
\end{equation}

As in the linear case, we consider three different regions; self-similar, elliptic, and hyperbolic.

The self-similar region is entirely similar to the linear case, so we omit the details.

The elliptic region is also almost similar, but there is a subtle difference due to the presence of nonlinearity. More precisely, we again split $g=g_{lo}+g_{hi},$ where the frequency scale is $R^{\frac{1}{4}}.$ Moreover, letting $v_1:=v-x^{-1}g_{lo}$ and \[\tilde{L}^{NL}v_1=(x+5\partial_x^4)v_1+uv_1=g_{hi}-5\partial_x^4(x^{-1}g_{lo}):=g_1.\]The estimation of $g_1$ is the same as in the linear case, yielding the $\epsilon R^{-\frac{1}{8}}$ bound. The only difference is the estimation of integrals. To be more specific, we start from the equality \[\int_{\mathbb{R}}v_1 \tilde{L}^{NL}v_1dx=\int_{\mathbb{R}}g_1v_1dx\] and integrating by parts gives \[\int_{\mathbb{R}}x|v_1|^2 dx+5\int_{\mathbb{R}}|\partial_x^2v_1|^2dx=\int_{\mathbb{R}}g_1v_1dx+\int_{\mathbb{R}}uv_1^2dx\]
The first integral on the right is routine, so we focus on the second integral including nonlinearity. A direct calculation gives \[\int_{\mathbb{R}}uv_1^2dx=O(M\epsilon R^{-\frac{3}{8}})\Vert v_1\Vert_{L^2}^2,\]so this bound can be absorbed into the left hand side provided $R$ is large enough. Note that by the time bound \eqref{rescaled} and the assumption $M\gg 1,$ the choice of $R$ is independent of $M$. Thus one may conclude \[\Vert v_1\Vert_{L^2}\lesssim \epsilon R^{-\frac{9}{8}},\quad \Vert\partial_x^2 v_1\Vert_{L^2}\lesssim \epsilon R^{-\frac{5}{8}}.\]
Now the same process as in the linear case gives the extra bounds \[\Vert \partial_x^4v_1\Vert_{L^2}\lesssim \epsilon R^{-\frac{1}{8}},\quad \Vert \partial_xv_1\Vert_{L^2}\lesssim \epsilon R^{-\frac{7}{8}},\quad \Vert \partial_x^3v_1\Vert_{L^2}\lesssim \epsilon R^{-\frac{3}{8}}.\]
Hence Gagliardo-Nirenberg inequality gives \[|\partial_x^k v_1|\lesssim\epsilon R^{-1+\frac{k}{4}},\quad k=0,1,2,3.\]
The pointwise bound of $x^{-1}g_{lo}$ can be obtained in the same way as in the linear analysis, so we omit the details.

Lastly, let us turn to the hyperbolic region. Again we do a spatial energy estimate to obtain \[\frac{d}{dx}(-x|v|^2+5|v_{2x}|^2-10v_xv_{3x}+2g_{lo}v)=-|v|^2+5uvv_x-2v_xg_{hi}+2g_{lo,x}v.\] Now integrating both sides and using H\"older's inequality gives \begin{align*}-x|v|^2+5|v_{2x}|^2-10v_xv_{3x}&\lesssim \Vert v_x\Vert_{L^2}\Vert g_{hi}\Vert_{L^2}+\Vert g_{lo,x}\Vert_{L^2}\Vert v\Vert_{L^2}+\sup_{x\in A_R^H}|g_{lo}v|+\int |uvv_x| dx\\&\lesssim \epsilon^2 R^{\frac{1}{4}}+M\epsilon^3 R^{\frac{1}{8}}\lesssim \epsilon^2 R^{\frac{1}{4}}.\end{align*}
A similar calculation gives \[-x|v_x|^2+5|v_{3x}|^2+2xvv_{2x}\lesssim \epsilon^2 R^{\frac{3}{4}},\]and the rest is now the same with the linear case. This completes the proof of the main theorem.
\begin{rmk}
Here one may see that when estimating on the hyperbolic region, one may estimate the nonlinear interaction term $uvv_x$ directly, unlike in the KdV case where one has to split the nonlinear interaction term and send some pieces to the left hand side. This difference is again due to the better spatial decay of our bootstrap bound \eqref{bootstrap} than in the case of KdV equation.
\end{rmk}
\vspace{10mm}

\section{Bounds for linearized equation}
In this section, we give the proof of Proposition 4.3.
\\Note that \eqref{eq:KW} enjoys scaling symmetry $u(t,x)\mapsto \lambda^4 u(\lambda^5t,\lambda x),$ which is critical in $\dot{H}^{-\frac{7}{2}}$ in the sense that the scaling preserves $\dot{H}^{-\frac{7}{2}}$ norm. The scaling vector field which generates the symmetry is given as $\mathcal{S}=5t\partial_t+x\partial_x+4.$

Define $\Lambda=\partial_x^{-1}\mathcal{S}$ and $\mathcal{L}=\partial_t-\partial_x^5.$ Then by the commutator relation \[[\mathcal{L},\mathcal{S}]=5\mathcal{L},\quad [\mathcal{S},\partial_x]=-\partial_x\] one has \begin{align*}
    \mathcal{L}(\Lambda u)&=\partial_x^{-1}(\mathcal{S}+5)\mathcal{L}u=u\partial_x\Lambda u.
\end{align*} Hence, $z=\Lambda u$ solves the linearized equation \[z_t-z_{5x}=uz_x.\] However, the function $z=xu+u_{4x}+\frac{5}{2}u^2+3\partial_x^{-1}u$ contains the inverse derivative of $u$, which can cause problem if $u$ is not mean-zero. Hence, we rather remove the $\partial_x^{-1}u$ term and work with the function \[w:=L^{NL}u:=xu+u_{4x}+\frac{5}{2}u^2,\] which solves the inhomogeneous linearized equation \[w_t+w_{5x}=uw_x-\frac{3}{2}u^2.\]
Now we are going to prove Proposition 4.3.
\\
Let $y=|D|^{\frac{1}{2}}z,$ then $y$ satisfies the equation \[y_t-y_{5x}=-|D|^{\frac{1}{2}}(uH|D|^{\frac{1}{2}}y).\] A direct calculation and frequency analysis as in \cite{ifrim2019dispersive} gives \[\frac{1}{2}\frac{d}{dt}\int y^2 = -\int |D|^{\frac{1}{2}}y\cdot uH|D|^{\frac{1}{2}}y=O(M\epsilon t^{-\frac{2}{5}})\Vert y\Vert_{L^2}^2-\int |D|^{\frac{1}{2}}y_{hi}\cdot u_{hi}H|D|^{\frac{1}{2}}y_{hi}\] in the time scale $M^2\epsilon t^{\frac{3}{5}}\leq 1.$ Here $y_{hi}:=P_{\gtrsim t^{-\frac{1}{5}}}y$ means a piece of $y$ frequency-localized in the scale $|\xi|\gtrsim t^{-\frac{1}{5}}.$

Hence, we shall try a normal form analysis to cancel the high frequency term, following the idea in \cite{hunter2015long}. Let $\tilde{u}=u+B(u,u)$ where $B$ is a symmetric bilinear Fourier multiplier. We want to find $B$ such that $\tilde{u}_t-\tilde{u}_{5x}$ consists of cubic or higher order terms of $u$. A direct calculation gives that \[\tilde{u}_t-\tilde{u}_{5x}=uu_x+B(u_{5x},u)+B(u,u_{5x})-\partial_x^5B(u,u)+h.o.t.,\]where $h.o.t.$ denotes the cubic or higher order terms of $u$. Here our goal is to cancel out $uu_x+B(u_{5x},u)+B(u,u_{5x})-\partial_x^5B(u,u).$ It turns out that the formal normal form is given as $\frac{1}{10}B(\partial_x^{-1}u,\partial_x^{-1}u)$ with the symbol of $B$ being $\hat{B}(\xi,\eta)=\frac{1}{\xi^2+\xi\eta+\eta^2}.$ However, this can cause problems on the low frequency region, so we truncate the low frequency part and get $\frac{1}{10}B(\partial_x^{-1}u_{hi},\partial_x^{-1}u_{hi}).$ Now linearizing the normal form gives the normal form correction of \eqref{eq:LinKW} as $\frac{1}{10}B(\partial_x^{-1}u_{hi},\partial_x^{-1}z_{hi}).$ Then plugging $y=|D|^{\frac{1}{2}}z$ gives $\frac{1}{10}B(\partial_x^{-1}u_{hi},\partial_x^{-1}|D|^{-\frac{1}{2}}y_{hi}).$ Let $\tilde{z}=z+\frac{1}{10}B(\partial_x^{-1}u_{hi},\partial_x^{-1}z_{hi}).$ Then observe that \begin{align*}\Vert \tilde{z}\Vert_{\dot{H}^{\frac{1}{2}}}&=\Vert y+|D|^{\frac{1}{2}}\frac{1}{10}B(\partial_x^{-1}u_{hi},\partial_x^{-1}|D|^{-\frac{1}{2}}y_{hi})\Vert_{L^2}^2\\=&\Vert y\Vert_{L^2}^2+\frac{1}{5}\langle |D|^{\frac{1}{2}}y,B(\partial_x^{-1}u_{hi},\partial_x^{-1}|D|^{-\frac{1}{2}}y_{hi})\rangle+\Vert \frac{1}{10}|D|^{\frac{1}{2}}B(\partial_x^{-1}u_{hi},\partial_x^{-1}|D|^{-\frac{1}{2}}y_{hi})\Vert_{L^2}^2.\end{align*}Here the last term is quartic, so it becomes more perturbative, and we do not include this into the corrected energy term. Hence the energy correction term is given as \[E'[y]=\frac{1}{10}\int |D|^{\frac{1}{2}}y_{hi}B(\partial_x^{-1}u_{hi},
\partial_x^{-1}|D|^{-\frac{1}{2}}y_{hi}).\] This corrected energy plays the desired role itself, but a more symmetric form would be convenient, so shifting $|D|^{-1}$ from the rightmost $y_{hi}$ to the leftmost one gives
\[E'[y]=\frac{1}{10}\int |D|^{-\frac{1}{2}}y_{hi}B(\partial_x^{-1}u_{hi},
H|D|^{-\frac{1}{2}}y_{hi})\]First check that $E[y]:=\frac{1}{2}\Vert y\Vert_{L^2}^2+E'[y]\sim \Vert y\Vert_{L^2}^2.$ This directly follows from:\[|E'[y]|\lesssim t^{\frac{1}{10}}\Vert y\Vert_{L^2}t^{\frac{2}{5}}t^{\frac{1}{5}}\Vert u\Vert_{L^\infty}t^{\frac{1}{10}}\Vert y\Vert_{L^2}\leq M\epsilon t^{\frac{3}{5}}\Vert y\Vert_{L^2}^2\leq \frac{1}{M}\Vert y\Vert_{L^2}^2,\]so by letting $M$ sufficiently large, we conclude $E[y]\sim \Vert y\Vert_{L^2}^2.$ Now a direct calculation gives\footnote{Here I intentionally dropped any insignificant constant multiplications.} \begin{align*}
    &\frac{d}{dt}E'[y]\\=&D_1+D_2+D_3+\int |D|^{-\frac{1}{2}}\partial_x^5y_{hi}B(\partial_x^{-1}u_{hi},
H|D|^{-\frac{1}{2}}y_{hi})\\+&\int |D|^{-\frac{1}{2}}y_{hi}B(\partial_x^{4}u_{hi},
H|D|^{-\frac{1}{2}}y_{hi})+\int |D|^{-\frac{1}{2}}y_{hi}B(\partial_x^{-1}u_{hi},
H|D|^{-\frac{1}{2}}\partial_x^5y_{hi})\\=&D_1+D_2+D_3-\int |D|^{-\frac{1}{2}}y_{hi}\partial_x^5B(\partial_x^{-1}u_{hi},
H|D|^{-\frac{1}{2}}y_{hi})\\+&\int |D|^{-\frac{1}{2}}y_{hi}B(\partial_x^{4}u_{hi},
H|D|^{-\frac{1}{2}}y_{hi})+\int |D|^{-\frac{1}{2}}y_{hi}B(\partial_x^{4}u_{hi},
H|D|^{-\frac{1}{2}}\partial_x^5y_{hi})
\\=&D_1+D_2+D_3+\int |D|^{-\frac{1}{2}}y_{hi} \partial_x(u_{hi} |D|^{\frac{1}{2}}y_{hi})=D_1+D_2+D_3+\int H|D|^{\frac{1}{2}}y_{hi} \cdot u_{hi} |D|^{\frac{1}{2}}y_{hi}
\end{align*} where the penultimate equality follows from the symbol calculation \[\frac{(\xi+\eta)^5-\xi^5-\eta^5}{\xi^2+\xi\eta+\eta^2}=5\xi\eta(\xi+\eta).\] Hence the last integral cancels out the ``bad" term of $\frac{1}{2}\frac{d}{dt}\int y^2$, and the other terms($D_1$, $D_2$, $D_3$) are perturbative, as demonstrated below:
\begin{enumerate}[(a)]
    \item $D_1$ arises from the time derivative of the frequency scale truncation, i.e. $\partial_t P_{hi}.$ Namely, $\partial_t(1-\varphi(t^{\frac{1}{5}}\xi))=-\frac{1}{5}t^{-\frac{4}{5}}\xi\varphi'(t^{\frac{1}{5}}\xi)=-t^{-1}\psi(t^{\frac{1}{5}}\xi)$ where $\psi(\eta)=\frac{1}{5}\eta \varphi'(\eta)$ is the Littlewood-Paley projection on the frequency scale $\sim 1.$ Thus the bound for $D_1$ is given as \begin{align*}
    &\left|\int |D|^{-\frac{1}{2}}(\partial_t P_{hi})y\cdot B(\partial_x^{-1}u_{hi},
H|D|^{-\frac{1}{2}}y_{hi})\right|\lesssim t^{-\frac{9}{10}}\Vert y\Vert_{L^2}\Vert B(\partial_x^{-1}u_{hi},
H|D|^{-\frac{1}{2}}y_{hi})\Vert_{L^2}\\\lesssim&t^{-\frac{9}{10}}\Vert y\Vert_{L^2}t^{\frac{2}{5}}\Vert\partial_x^{-1}u_{hi}\Vert_{L^\infty}\Vert H|D|^{-\frac{1}{2}}y_{hi}\Vert_{L^2}\lesssim M\epsilon t^{-\frac{2}{5}}\Vert y\Vert_{L^2}^2.
    \end{align*} 
    The rest terms, i.e. $\int |D|^{-\frac{1}{2}}y_{hi}\cdot B(\partial_x^{-1}(\partial_t P_{hi})u,
H|D|^{-\frac{1}{2}}y_{hi})$ and \\$\int |D|^{-\frac{1}{2}}y_{hi}\cdot B(\partial_x^{-1}u_{hi},
H|D|^{-\frac{1}{2}}(\partial_t P_{hi})y)$ can be estimated in a similar fashion.
    \item $D_2$ consists of the nonlinarity of $u$ from \eqref{eq:KW}:\begin{align*}
    \left|\int |D|^{-\frac{1}{2}}y_{hi}B((u^2)_{hi},
H|D|^{-\frac{1}{2}}y_{hi})\right|\lesssim t^{\frac{1}{10}}\Vert y\Vert_{L^2}t^{\frac{2}{5}}\Vert u^2\Vert_{L^\infty}t^{\frac{1}{10}}\Vert y\Vert_{L^2} \leq M\epsilon t^{-\frac{2}{5}}.
    \end{align*}
    \item $D_3$ consists of the linearized counterpart of $D_2$:
    \begin{align*}
    &\left|\int (uH|D|^{\frac{1}{2}}y)_{hi}\cdot B(\partial_x^{-1}u_{hi},H|D|^{-\frac{1}{2}}y_{hi})\right|\\=&\left|\int Hy\cdot |D|^{\frac{1}{2}}\left(uP_{hi} B(\partial_x^{-1}u_{hi},H|D|^{-\frac{1}{2}}y_{hi})\right)\right|\lesssim\Vert y\Vert_{L^2} \Vert|D|^{\frac{1}{2}}\left(uP_{hi} B(\partial_x^{-1}u_{hi},H|D|^{-\frac{1}{2}}y_{hi})\right)\Vert_{L^2}^2.
    \end{align*} Now the Kato-Ponce commutator estimate (see, for example, \cite{kenig1993well}) \[\Vert |D|^s(fg)-(|D|^sf)g \Vert_{L^2}\lesssim \Vert f\Vert_{L^\infty}\Vert |D|^sg\Vert_{L^2},\quad 0<s<1, \] gives\[\Vert |D|^s(fg)\Vert_{L^2}\lesssim \Vert |D|^s(fg)-(|D|^sf)g\Vert_{L^2}+\Vert (|D|^sf)g\Vert_{L^2}\lesssim \Vert f\Vert_{L^\infty}\Vert |D|^sg\Vert_{L^2}+\Vert |D|^sf\Vert_{L^\infty}\Vert g\Vert_{L^2},\]so that\begin{align*}
    &\Vert|D|^{\frac{1}{2}}\left(uP_{hi} B(\partial_x^{-1}u_{hi},H|D|^{-\frac{1}{2}}y_{hi})\right)\Vert_{L^2}\\\lesssim&  \Vert|D|^{\frac{1}{2}}u\Vert_{L^\infty}\Vert P_{hi} B(\partial_x^{-1}u_{hi},H|D|^{-\frac{1}{2}}y_{hi})\Vert_{L^2}+\Vert u\Vert_{L^\infty}\Vert |D|^{\frac{1}{2}}P_{hi} B(\partial_x^{-1}u_{hi},H|D|^{-\frac{1}{2}}y_{hi})\Vert_{L^2}\\\lesssim& M\epsilon t^{-\frac{3}{10}}\cdot t^{\frac{2}{5}}\Vert\partial_x^{-1}u_{hi}\Vert_{L^\infty}\Vert H|D|^{-\frac{1}{2}}y_{hi}\Vert_{L^2}+M\epsilon t^{-\frac{1}{5}}\cdot t^{\frac{3}{10}}\Vert\partial_x^{-1}u_{hi}\Vert_{L^\infty}\Vert H|D|^{-\frac{1}{2}}y_{hi}\Vert_{L^2}\\\lesssim&M\epsilon t^{-\frac{2}{5}}\Vert y\Vert_{L^2}.
    \end{align*} Hence we obtain the bound $M\epsilon t^{-\frac{2}{5}}\Vert y\Vert_{L^2}^2$. The other term $\int |D|^{-\frac{1}{2}}y_{hi}B(\partial_x^{-1}u_{hi},
H(u|D|^{\frac{1}{2}}y)_{hi})$ can be treated similarly after decomposing $(u|D|^{\frac{1}{2}}y)_{hi}=\partial_x(uH|D|^{-\frac{1}{2}}y)_{hi}-(u_xH|D|^{-\frac{1}{2}}y)_{hi}$. One may handle the former term by noting that the bilinear map $(u,v)\mapsto B(u,v_x)$ has symbol $\frac{i\eta}{\xi^2+\xi\eta+\eta^2}$, so that $\Vert B(u,v_x)\Vert_{L^2}\lesssim t^{\frac{1}{5}}\Vert u\Vert_{L^\infty}\Vert v\Vert_{L^2},$ provided $u$ and $v$ are frequency supported at the scale $\gtrsim t^{-\frac{1}{5}}.$
\\
Hence, we conclude that, if we let the corrected energy $E[y]=\frac{1}{2}\int y^2 + E'[y],$ then \[\frac{d}{dt}E[y]\lesssim M\epsilon t^{-\frac{2}{5}}\Vert y\Vert_{L^2}^2\sim M\epsilon t^{-\frac{2}{5}} E[y].\] Hence Gronwall's inequality gives \[E[y]\leq e^{CM\epsilon t^{\frac{3}{5}}}E[y_0]\lesssim E[y_0]\] in the time scale $t\ll_M \epsilon^{-\frac{5}{3}}.$ This completes the proof of Proposition 4.3.
\end{enumerate}
\begin{rmk}\begin{enumerate}[(1)]
\item The above analysis also works when the right hand side of \eqref{eq:LinKW} is $u^k w_x$ with $k\geq 1,$ with $u_{hi}$ being replaced by $(u^k)_{hi}.$ Namely, we may consider $u^k$ as a single component, which makes no difference of the details. The only difference is the valid time scale of the analysis. For instance, if $k=2,$ which corresponds to the case of \eqref{eq:5KdV} with quadratic terms gone, the time scale becomes $\epsilon^{-5}.$
\item The analysis in this section is almost identical to that in \cite{ifrim2019dispersive}, since the KdV and Kawahara equation share the same form of nonlinear terms and their linearized counterparts, on which the overall analysis mainly depends. The linear part only affects the form of the corrected energy.
\end{enumerate}\end{rmk}
\section{Sobolev bound of $L^{NL}u$}
In this section, we are going to prove Proposition 4.4. Unlike the \cite{ifrim2019dispersive}, we do not have to dive into a complicated normal form analysis, since we have a better bootstrap bound for $u$. Namely, we already proved that $\Vert u^2\Vert_{\dot{H}^{\frac{1}{2}}}\lesssim M^2\epsilon^2t^{-\frac{2}{5}}$ in the Proposition 4.1. Then it remains to see that the $\dot{H}^{\frac{1}{2}}$ bound of the inhomogeneous term is perturbative enough that the result of Proposition 4.3 is still valid to $L^{NL}u.$ Of course, the answer is affirmative, which is merely the Proposition 4.4.\\
Let $z:=|D|^{\frac{1}{2}}L^{NL}u.$ Then $z$ solves the equation \[z_t-z_{5x}=-|D|^{\frac{1}{2}}(uH|D|^{\frac{1}{2}}z-\frac{3}{2}u^2).\] Hence one has:
\begin{align*}
\frac{d}{dt}E[z]=&-\frac{3}{2}\int z|D|^{\frac{1}{2}}(u^2)- \frac{3}{10}\int (u^2)_{hi}B(\partial_x^{-1}u_{hi},H|D|^{-\frac{1}{2}}z_{hi})\\ -&\frac{3}{10}\int |D|^{-\frac{1}{2}}z_{hi}B(\partial_x^{-1}u_{hi},H(u^2)_{hi})+O(M\epsilon t^{-\frac{2}{5}})\Vert z\Vert_{L^2}^2,
\end{align*}
where $E[\cdot]$ is the corrected energy defined in the Section 5. Here most of the steps are identical to the Section 5, so it only remains to estimate the integrals including the inhomogeneous term. Namely, \begin{align*}&\left|\int z|D|^{\frac{1}{2}}(u^2)\right|\lesssim M^2\epsilon^2t^{-\frac{2}{5}}\Vert z\Vert_{L^2}\lesssim M^2\epsilon t^{-\frac{2}{5}}\Vert z\Vert_{L^2}^2+M^2\epsilon^3t^{-\frac{2}{5}}\lesssim M^2\epsilon t^{-\frac{2}{5}}\Vert z\Vert_{L^2}^2+\epsilon^2 t^{-1},\\&\left|\int (u^2)_{hi}B(\partial_x^{-1}u_{hi},H|D|^{-\frac{1}{2}}z_{hi})\right|\\\lesssim &\Vert u^2\Vert_{\dot{H}^{\frac{1}{2}}}t^{\frac{2}{5}}\Vert \partial_x^{-1}u_{hi}\Vert_{L^\infty}\Vert |D|^{-\frac{1}{2}}z_{hi}\Vert_{L^2}\lesssim M^3\epsilon^3 t^{\frac{1}{10}}\Vert z\Vert_{L^2}\lesssim M^2\epsilon t^{-\frac{2}{5}}\Vert z\Vert_{L^2}^2+M^4\epsilon^5 t^{\frac{3}{5}}\\\lesssim&M^2\epsilon t^{-\frac{2}{5}}\Vert z\Vert_{L^2}^2+\epsilon^2 t^{-\frac{6}{5}},\end{align*}
and the rest terms can be estimated similarly.

To sum up, one obtains the bound 
\[\frac{d}{dt}E[z]\lesssim M^2\epsilon t^{-\frac{2}{5}}\Vert z\Vert_{L^2}^2+\epsilon^2 t^{-1}.\]
Hence Gr\"onwall's inequality gives \[\Vert z\Vert_{L^2}^2\lesssim \epsilon^2 e^{CM^2\epsilon t^{\frac{3}{5}}}\lesssim\epsilon^2\] in the time scale $|t|\ll\epsilon^{-\frac{5}{3}}.$ This completes the proof of Proposition 4.4.

\section{Appendix: The case of modified Kawahara equation} In this appendix, we sketch the proof of the case $m=2$, namely \eqref{eq:mKW}, of the Theorem 1.1.\\First, when it comes to the Subsection 4.1, one may rescale the equation \[xu+5tu_{4x}+\frac{5}{3}tu^3=f,\quad f=L^{NL}u\] to \[x\tilde{u}+5\tilde{u}_{4x}+\frac{5}{3}\tilde{u}^3=\tilde{f}\] by letting $\tilde{u}(x)=t^{\frac{2}{5}}u(t,xt^{\frac{1}{5}})$ and $\tilde{f}(x)=t^{\frac{1}{5}}f(t,xt^{\frac{1}{5}})$. Also the localized equation becomes \[(x+5\partial_x^4)v+\frac{5}{3}u^2v=g\] with $v$ and $g$ satisfying the same bounds as in the Subsection 4.1. In these settings, all the same steps work properly.\\Moreover, one may easily see that the argument in the Section 5 does work to prove Proposition 4.1 in the case of \eqref{eq:mKW}, under the time bound assumption $M^{\frac{3}{2}}\epsilon t^{\frac{1}{5}}\leq 1.$ Here is the step where the time scale restriction $|t|\ll \epsilon^{-5}$ is needed.\\Secondly, as mentioned in the remark at the end of the Section 5, one may prove the same bound from Proposition 4.3 with the same argument. Here we have a subtly different setting, and I state here for convenience.\\The equation \eqref{eq:mKW} enjoys the scaling symmetry $u(t,x)\mapsto \lambda^2u(\lambda^5t,\lambda x)$, so the scaling vector field is given as $\mathcal{S}=5t\partial_t+x\partial_x+2,$ and one may see that if $u$ solves \eqref{eq:mKW}, then $\Lambda u:=\partial_x^{-1}\mathcal{S}u=xu+5tu_{4x}+\frac{5}{3}tu^3+\partial_x^{-1}u$ solves the linearized equation \[w_t-w_{5x}=u^2w_x.\] Now the rest of the arguments are the same as in the Section 5.

The Section 6 is again similar. In this case $L^{NL}u=xu+5tu_{4x}+\frac{5}{3}tu^3=\Lambda u-\partial_x^{-1}u$ solves the inhomogeneous equation \[z_t-z_{5x}=u^2z_x+\frac{2}{3}u^3,\] and the energy estimate of the inhomogeneous equation is now routine. Note that the sign of $c$ of \eqref{eq:mKW} does not play any role in the entire steps of the proof.
\section{Conclusion}
We have discussed how long does a small data solution to Kawahara and modified Kawahara equation shows a linear dispersive decay bound. The result shows the tendency of change of the time scale where the linear dispersive decay bound holds depending on the linear and nonlinear part of the equation. More specifically, for the Kawahara equation, which has the same nonlinearity as KdV equation and a higher order of the linear part, the time scale of linear dispersive decay bound is possibly shorter than that of KdV equation. On the other hand, for the modified Kawahara equation, which has a higher order nonlinearity than KdV and Kawahara equations, has a longer time scale of such decay bound. However, unlike the KdV equation, we still do not know whether the time bounds found in this article is optimal due to the non-integrability of the equations, which should be found in a future work.
\section*{Declarations}
\noindent\\\small{\textbf{Acknowledgements}\\The author appreciate Soonsik Kwon for helpful discussions and encouragement to this work.}
\\\small{\textbf{Funding statement}\\The author is partially supported by NRF-2019R1A5A1028324 and NRF-2018R1D1A1A0908335.}
\\\small{\textbf{Availability of data and materials}\\Not applicable.}
\\\textbf{Competing interests}\\The author declares that he does not have any competing interests.\\
\textbf{Author's contributions}\\The single author contributed and reviewed the work solely.
\printbibliography
\end{document}